\documentclass[12pt]{amsart}
\usepackage{epsfig}
\usepackage{mathtools}
\usepackage[subnum]{cases}
\usepackage{enumerate}
\usepackage{bbm,bm}
\usepackage{amsmath,amssymb,mathrsfs}
\usepackage{color}
\newtheorem{theorem}{Theorem}[section]

\newtheorem{lemma}[theorem]{Lemma}
\newtheorem{corollary}[theorem]{Corollary}
\newtheorem{remark}[theorem]{Remark}
\newtheorem{proposition}[theorem]{Proposition}

\newcommand{\R}{\mathbb{R}} 

\newcommand{\fconc}{{\mbox{\rm Conc}(0,\infty)}}
\newcommand{\fconv}{{\mbox{\rm Conv}(0,\infty)}}
\newcommand{\fconcstargen}{{\mbox{\rm Conc}^{-}(0,\infty)}}
\newcommand{\fconvstargen}{{\mbox{\rm Conv}(0,\infty)}}

\newcommand{\fconvstar}{{\mbox{\rm Conv}(0,\infty)}}

\newcommand{\fconcstardualgen}{\mbox{\rm Conc}^+(0,\infty)}

\newcommand{\K}{\mathcal{K}_o^n}

\newcommand{\ISphi}{{\mbox{\rm IS}_\varphi}}
\newcommand{\iSphi}{{\mbox{\rm is}_\varphi}}
\newcommand{\OSphi}{{\mbox{\rm OS}_\varphi}}
\newcommand{\oSphi}{{\mbox{\rm os}_\varphi}}

\newcommand{\ISpsi}{{\mbox{\rm IS}_\psi}}
\newcommand{\iSpsi}{{\mbox{\rm is}_\psi}}
\newcommand{\OSpsi}{{\mbox{\rm OS}_\psi}}
\newcommand{\oSpsi}{{\mbox{\rm os}_\psi}}

\newcommand{\OSphistar}{{\mbox{\rm OS}_{\varphi^*}}}

\newcommand{\ISpsistar}{{\mbox{\rm IS}^*_\psi}}
\newcommand{\iSpsistar}{{\mbox{\rm is}^*_\psi}}
\newcommand{\OSpsistar}{{\mbox{\rm OS}^*_\psi}}
\newcommand{\oSpsistar}{{\mbox{\rm os}^*_\psi}}

\newcommand{\oSp}{{\mbox{\rm os}_{p}}}
\newcommand{\iSpstar}{{\mbox{\rm is}_{p}^*}}

\newcommand{\ISp}{{\mbox{\rm IS}_{p}}}
\newcommand{\OSp}{{\mbox{\rm OS}_{p}}}

\DeclareMathOperator{\vrad}{vrad}
\DeclareMathOperator{\as}{as}
\newcommand{\Cin}{\mathscr{C}_{{\rm in}}}
\newcommand{\Cout}{\mathscr{C}_{{\rm out}}}

\begin{document}

\title{Extremal general  affine surface areas} 
\date{\today}

    \author[Steven Hoehner]{Steven Hoehner}

\maketitle

\begin{abstract}
For a convex body $K$ in $\R^n$, we introduce and study the extremal general affine surface areas, defined by
\[
\ISphi(K):=\sup_{K^\prime\subset K}\as_\varphi(K),\quad \oSpsi(K):=\inf_{K^\prime\supset K}\as_\psi(K)
\]
where $\as_\varphi(K)$ and $\as_\psi(K)$ are the $L_\varphi$ and $L_\psi$ affine surface area of $K$, respectively. We prove that  there exist extremal convex bodies that achieve the supremum and infimum, and that the functionals $\ISphi$ and $\oSpsi$ are continuous. In our main results, we prove  Blaschke-Santal\'o type inequalities and inverse Santal\'o type inequalities for the  extremal general affine surface areas. This article may be regarded as an Orlicz extension of the recent work of Giladi, Huang, Sch\"utt and Werner (2020), who introduced and studied the extremal $L_p$ affine surface areas.
\end{abstract}


\section{Introduction}

The celebrated result of Fritz John \cite{John1948} says that every convex body $K$ in $\R^n$  contains a unique ellipsoid of maximum volume,  now called the John ellipsoid of $K$. The minimum volume ellipsoid containing $K$ is also unique and  is called the L\"owner ellipsoid of $K$. The John and L\"owner ellipsoids are related by polar duality: If $K$ contains the origin in its interior, then the John ellipsoid is the polar body of the L\"owner ellipsoid, and vice versa. These ellipsoids are cornerstones of asymptotic convex geometry, arising in concentration of volume, the reverse isoperimetric inequality, the hyperplane conjecture and many more \cite{AGA-book, isotropic-book}. Other prominent applications of the John and L\"owner ellipsoids can be found in Banach space geometry \cite{BourgainMilman, FLM, Pisier1989, Szarek1978},  extremal problems \cite{Ball1991,Barthe1998, GM2000, Schmuck1999}, polytopal approximation of convex bodies \cite{Gruber1993, Ludwig1999, SW2003},  linear programming \cite{GLS1993, TY1997} 
and statistics \cite{Davies92, Rousseeuw1985,VanAelstRousseeuw}. 


In this article, we study an analogue of John's theorem when volume  is replaced by the general  $L_\varphi$ or $L_\psi$ affine surface area, which is defined in terms of a special type of concave function $\varphi$ or convex function $\psi$, respectively.  More specifically, instead of considering the ellipsoid of maximum (resp. minimum) volume contained in (containing) $K$, we will consider the convex body of maximum $L_\varphi$ (minimum $L_\psi$) affine surface contained in (containing) $K$. 
When $K$ has centroid at the origin, we define the {\it inner maximal $L_\varphi$ affine surface area} $\ISphi(K)$ and {\it outer minimal $L_\psi$ affine surface area} $\oSpsi(K)$ of $K$ by
\begin{align}
    \ISphi(K)&:=\sup_{K^\prime \subset K}\as_\varphi(K^\prime)\label{defn1}\\
    \oSpsi(K)&:=\inf_{K^\prime\supset K}\as_\psi(K^\prime)\label{defn2}.
\end{align}
Here $\as_\varphi(K)$ and $\as_\psi(K)$ denote the general $L_\varphi$ and $L_\psi$ affine surface area of $K$, respectively, which are extensions of the classical affine surface area defined by Sch\"utt and Werner \cite{SW1990}. The supremum and infimum are taken over all $K^\prime$ with centroid at the origin  (see  Section \ref{genASsect} for the definitions).

The special case of \eqref{defn1} when $\as_\varphi$ is replaced by the classical affine surface area has attracted considerable interest. Inscribed bodies of maximal affine surface area have found  applications to discrete geometry \cite{Barany1995, Barany97, Barany1999, BP2006}, geometric probability \cite{Barany1999} and variational problems in differential equations \cite{LiuZhou2013, TW2004}. The existence of a convex body $K_{\max}\subset K$ of maximum affine surface area follows from the upper semicontinuity of the affine surface area \cite{Barany97, Schneider2014}. The regularity properties of $K_{\max}$ were studied by Sheng, Trudinger and Wang \cite{STW2006}. For planar convex bodies, B\'ar\'any proved that $K_{\max}$  is unique, and showed a remarkable relationship between $K_{\max}$ and the limit shape of the convex hull of certain lattice points \cite{Barany97, BP2006} or random points \cite{Barany1999} contained in $K$. B\'ar\'any and Prodromou \cite{BP2006} showed that if $K$ is a planar convex body, then $K_{\max}$ is of elliptic type, and if $K$ is of elliptic type then $K_{\max}=K$.  Schneider \cite{Schneider2014} used a new method to  extend the latter result to all dimensions, showing that if $K$ is a convex body in $\R^n$ of elliptic type, then $K_{\max}=K$.

Recently, Giladi, Huang, Sch\"utt and Werner \cite{Giladietal} introduced and studied the extremal $L_p$ affine surface areas
\begin{align}
    \ISp(K)&:=\sup_{K^\prime \subset K}\as_p(K^\prime)\label{defn1Lp}\\
    \oSp(K)&:=\inf_{K^\prime\supset K}\as_p(K^\prime)\label{defn2Lp}
\end{align}
where $\as_p(K)$ is the $L_p$ affine surface area of $K$. It was shown in \cite{Giladietal} that for any convex body $K$ with centroid at the origin, the extremal bodies exist, and continuity and  affine isoperimetric inequalities were proved  for the functionals $\ISp$ and $\oSp$.  Asymptotic estimates were also given  for $\ISp(K)$ and $\oSp(K)$ in terms of powers of the volume of $K$ using the L\"owner ellipsoid and thin shell estimate of \cite{GM2011}. 

  When a certain concave or convex function depending on $n$ and $p$ is chosen in \eqref{defn1} or \eqref{defn2}, respectively, we recover \eqref{defn1Lp} and \eqref{defn2Lp} as special cases. Thus, the new definitions \eqref{defn1} and \eqref{defn2} may be thought of as Orlicz extensions of the definitions \eqref{defn1Lp} and \eqref{defn2Lp}  (see Section \ref{Lpsection} below for the details). A key difference between the $L_p$ affine surface areas and the general $L_\varphi$ and $L_\psi$ affine surface areas is that the former are homogeneous, while the latter, in general, are not.  
  For our purposes, the lack of homogeneity will only present difficulties for the $L_\varphi$ affine surface areas. To surmount this obstacle, we will restrict our attention to certain subclasses of concave functions $\varphi$ satisfying mild growth rate conditions. These subclasses  contain many  functions that are not homogeneous of any degree, which  distinguishes the results in this work from those in  \cite{Giladietal}. At the same time, these classes contain the homogeneous functions considered in the $L_p$ setting of \cite{Giladietal}.

In the present paper, we prove existence and continuity of the extremal  general affine surface areas \eqref{defn1} and \eqref{defn2}. 
Our main result is a  Blaschke-Santal\'o type inequality for the inner maximal $L_\varphi$ affine surface area. More specifically, in Theorem \ref{BSvarphi} we show that for any convex body $K$ in $\R^n$ with centroid at the origin and any concave function $\varphi$ satisfying some prescribed conditions,
\begin{equation}\label{mainresult}
\ISphi(K)\ISphi(K^\circ)\leq \ISphi(B_n)^2
\end{equation}
with equality if and only if $K$ is an ellipsoid (here $B_n$ denotes the Euclidean unit ball $\R^n$ centered at the origin). Finally, we prove an inverse Santal\'o type inequality for the outer minimal $L_\psi$ affine surface area.

\subsection{Overview}

In Section \ref{Backgroundsec}, we state definitions and notation used throughout the paper,  and provide the relevant background on the general $L_\varphi$ and $L_\psi$ affine surface areas. In Section \ref{dualitysection}, we use a polar duality relation of Ludwig \cite{Ludwig2010} for the $L_\varphi$ affine surface areas to define the subclasses of concave functions $\varphi$ we will consider. Next, in Section \ref{defsection} we define the extremal general affine surface areas and state their existence, monotonicity and continuity properties; for the reader's convenience, the proofs of these properties are included at the end of the paper in Section \ref{proofs}. The lemmas that will be used in the proofs of the main results are in Section \ref{lemmas}, and  in Section \ref{isoperimetry} we prove affine isoperimetric inequalities for the extremal general affine surface areas. In Subsections \ref{BSsect} and \ref{invsantsec} we prove our main results, which are Blaschke-Santal\'o type inequalities and inverse Santal\'o type inequalities for the extremal general affine surface areas. 


\section{Background and notation}\label{Backgroundsec}

The standard inner product of $x,y\in\R^n$ is denoted $\langle x,y\rangle$. The Euclidean unit ball in $\R^n$ centered at the origin $o$ is the set  $B_n=\{x\in\R^n: \langle x,x\rangle \leq 1\}$. The Minkowski sum $A+B$ of two sets $A,B\subset \R^n$ is defined by $A+B=\{a+b: a\in A, \, b\in B\}$.  We say that a function $f:(0,\infty)\to (0,\infty)$ is {\it homogeneous of degree $r$} if $f(\lambda t)=\lambda^r f(t)$ for all $\lambda,t>0$.

A {\it convex body} in $\R^n$ is a convex, compact set with nonempty interior. Denote the class of convex bodies in $\R^n$ that contain the origin in their interiors by $\K$. For convex bodies $K$ and $L$ in $\R^n$, the {\it Hausdorff distance} $\delta_H(K,L)$ of $K$ and $L$ is defined as 
\[
\delta_H(K,L) = \inf\{\varepsilon\geq 0: K\subset L+\varepsilon B_n,\, L\subset K+\varepsilon B_n\}.
\]

The $n$-dimensional volume of $K$ is $|K|=\int_{\R^n}\mathbbm{1}_K(x)\,dx$. The {\it volume radius} $\vrad(K)$ of $K$ is the radius of the Euclidean ball with the same volume as $K$, that is, $\vrad(K)=(|K|/|B_n|)^{1/n}$. The centroid $g(K)$ of $K$ is defined by $g(K)=|K|^{-1}\int_K x\,dx$. If the centroid of $K$ is not the origin, then we translate $K$ so that it is. 
If $K\in\mathcal{K}_o^n$, then the {\it polar body} $K^\circ$ of $K$ is the convex body $K^\circ=\{x\in\R^n: \langle x,y\rangle \leq 1, \, \forall y\in K\}$. The bipolar theorem states that  $(K^\circ)^\circ=K$. The polarity operation is inclusion-reversing, i.e., $L\subset K$ if and only if $L^\circ\supset K^\circ$. 

We also let $\partial K$ denote the boundary of $K$.   In particular, the Euclidean unit ball has surface area $|\partial B_n|=n|B_n|$. The Gaussian curvature of $K$ at $x\in\partial K$ is denoted $\kappa(K,x)$, $\mu_{\partial K}$ is the usual surface measure on $\partial K$ and $N_{\partial K}(x)$ is the outer unit normal at $x\in\partial K$. For more background on convex geometry, we refer the reader to, e.g., the monograph of Schneider \cite{SchneiderBook}.

\subsection{General $L_\varphi$ and $L_\psi$ affine surface areas}\label{genASsect}

 In the past decade or so, there has been tremendous interest  in extending results from the $L_p$  theory of convex bodies to the rapidly growing  Orlicz  theory.  Prominent examples  include  the Orlicz-Petty projection inequality \cite{LYZ2010-1} and  Orlicz centroid inequalities \cite{CZY2011, LiLeng2011, LYZ2010-2, Zhu2012}, the Orlicz-Brunn-Minkowski theory \cite{GHXY2020, XJL2014}, the Orlicz-Minkowski problem \cite{HLYZ2010, HuangHe2012, JianLu2019},   Orlicz affine isoperimetric inequalities \cite{UmutDeping, Deping2012, Deping2014Steiner, Deping2015} and Orlicz-John ellipsoids \cite{ZX2014}. The $L_\varphi$ and $L_\psi$ general affine surface areas, which are generalizations of the classical affine surface area,  have played a key role in the initiation and development of this program. 

The $L_\varphi$ affine surface area has its origins in valuation theory, arising in a fundamental  result on the classification of upper semicontinuous $\text{SL}(n)$ invariant valuations on $\K$. In the groundbreaking work \cite{LudwigReitzner}, Ludwig and Reitzner proved that a  functional $\Phi:\K\to\R$ is an upper semicontinuous and $\text{SL}(n)$ invariant valuation that vanishes on polytopes in $\K$ if and only if there exists a concave function $\varphi:[0,\infty)\to[0,\infty)$ with $\lim_{t\to 0}\varphi(t)=\lim_{t\to\infty}\varphi(t)/t=0$ such that  
\begin{equation}\label{valdef}
\Phi(K)= \int_{\partial K}\varphi\left(\frac{\kappa(K,x)}{\langle x,N_{\partial K}(x)\rangle^{n+1}}\right)\,\langle x,N_{\partial K}(x)\rangle\,d\mu_{\partial K}(x)
\end{equation}
for every $K\in\K$. The integral on the right-hand side of \eqref{valdef} is called the {\it $L_\varphi$ affine surface area} of $K$ and is denoted by $\as_\varphi(K)$. That is,
\begin{equation}\label{genas1}
    \as_\varphi(K) := \int_{\partial K}\varphi\left(\frac{\kappa(K,x)}{\langle x,N_{\partial K}(x)\rangle^{n+1}}\right)\,\langle x,N_{\partial K}(x)\rangle\,d\mu_{\partial K}(x).
\end{equation}

In view of this result, let $\fconc$ denote the class of all concave functions $\varphi:(0,\infty)\to(0,\infty)$ such that  $\lim_{t\to 0}\varphi(t)=\lim_{t\to\infty}\frac{\varphi(t)}{t}=0$, and set $\varphi(0)=0$. (Examples of such functions include $\varphi(t)=t^\alpha$ for $\alpha\in(0,1)$,  $\varphi(t)=\log(t+1)$ and $\varphi(t)=\arctan(t)$.) When $\varphi(t)=t^{\frac{p}{n+p}}$ for $p>0$, one recovers the $L_p$ affine surface area; further choosing $p=1$ so that $\varphi(t)=t^{\frac{1}{n+1}}$, one obtains the classical affine surface area (see Subsection \ref{Lpsection}).

The $L_\psi$ affine surface area was introduced and studied in \cite{Ludwig2010} as a generalization of the $L_p$ affine surface area for $p\in(-n,0)$. Let $\fconv$ denote the class of all convex functions $\psi:(0,\infty)\to(0,\infty)$ such that  $\lim_{t\to 0}\psi(t)=\infty$ and $\lim_{t\to\infty}\psi(t)=0$, and set $\psi(0)=\infty$. (Examples of such functions include $\psi(t)=t^\alpha$ for $\alpha\in (-\infty,0)$ and $\psi(t)=\log(t^{-1}+1)$.)   
For $K\in\mathcal{K}_o^n$ and $\psi\in\fconv$, the {\it $L_\psi$ general affine surface area} $\as_\psi(K)$ is defined
by \cite{Ludwig2010}
\begin{equation}\label{defaspsi}
    \as_\psi(K) := \int_{\partial K}\psi\left(\frac{\kappa(K,x)}{\langle x,N_{\partial K}(x)\rangle^{n+1}}\right)\,\langle x,N_{\partial K}(x)\rangle\,d\mu_{\partial K}(x).
\end{equation}
Note that if $P\in\K$ is a polytope, then $\kappa(P,x)=0$ for almost all $x\in\partial P$, so  $\as_\psi(P)=\infty$. When $\psi(t)=t^{\frac{p}{n+p}}$ and $p\in(-n,0)$, one recovers the $L_p$ affine surface area  (see Subsection \ref{Lpsection}). 

As mentioned before, for $\varphi\in\fconc$ the   $L_\varphi$ affine surface area is {\it upper semicontinuous}, meaning for any sequence of convex bodies $\{K_j\}_{j\in\mathbb{N}}\subset\K$ that converges to $K$ with respect to the Hausdorff metric,
\begin{equation}\label{uppersemi}
\limsup_{j\to\infty}\as_\varphi(K_j) \leq \as_\varphi(K).
\end{equation}
 For $\psi\in\fconv$, Ludwig \cite{Ludwig2010} proved that the $L_\psi$ affine surface area  is {\it lower semicontinuous}, meaning for any sequence of convex bodies $\{K_j\}_{j\in\mathbb{N}}\subset\K$ that converges to $K$ with respect to the Hausdorff metric,
\begin{equation}\label{lowersemi}
\liminf_{j\to\infty}\as_\psi(K_j) \geq \as_\psi(K).
\end{equation}
The $L_\psi$ affine surface area  is also $\text{SL}(n)$ invariant \cite{Ludwig2010}.

The $L_\varphi$ and $L_\psi$ affine surface areas satisfy the following powerful affine isoperimetric inequalities due to Ludwig \cite{Ludwig2010}.

\begin{lemma}\cite[Thm. 3]{Ludwig2010}\label{isoperlemphi}
Let $K$ be a convex body in $\R^n$ with centroid at the origin. Then for any $\varphi\in\fconc$,
\begin{align*}
    \as_\varphi(K) &\leq \as_\varphi(\vrad(K)B_n). \end{align*}
Moreover, there is equality for strictly increasing $\varphi$ if and only if $K$ is an ellipsoid.
\end{lemma}

\begin{lemma}\cite[Thm. 8]{Ludwig2010}\label{isoperlem2}
Let $K$ be a convex body in $\R^n$ with centroid at the origin. Then for any $\psi\in\fconv$,
\begin{align*}
    \as_\psi(K) &\geq \as_\psi(\vrad(K)B_n). \end{align*}
Moreover, there is equality for strictly decreasing $\psi$ if and only if $K$ is an ellipsoid.
\end{lemma}
\section{Polar duality and the functional setting}\label{dualitysection}

If $K=rB_n$ for some positive number $r$, then by definitions \eqref{genas1} and \eqref{defaspsi} we obtain
\begin{align}
    \as_\varphi(rB_n)
    &=r^n\varphi(r^{-2n})|\partial B_n| \label{asball}\\
    \as_\psi(rB_n)
    &=r^n\psi(r^{-2n})|\partial B_n|.\label{asball2}
\end{align}
Since $\psi$ is decreasing, the $L_\psi$ affine surface areas are increasing on Euclidean balls.  More specifically, 
\begin{equation}\label{monotonepsi}
0<r\leq s \quad \Longrightarrow \quad \as_\psi(rB_n)\leq\as_\psi(sB_n). 
\end{equation}
This is not the case for the $L_\varphi$ affine surface area and a general  function $\varphi\in\fconc$. Fortunately,  we can control the monotonicity of $\as_\varphi$ on Euclidean balls  by restricting the growth rate of the ratio of $\varphi$ and the square root function. We will accomplish this via the following polar duality relation of Ludwig \cite{Ludwig2010}. For $\varphi\in\fconc$, define  $\varphi^*(t):=t\varphi(1/t)$. Then $\varphi^*$ is also an element of $\fconc$. Ludwig \cite[Thm. 4]{Ludwig2010} proved that for every $K\in\K$, 
\begin{equation}\label{duality1}
\as_{\varphi}(K)=\as_{\varphi^*}(K^\circ).
\end{equation}
In this way, the operation $\varphi\mapsto \varphi^*$ induces a duality relationship within the class $\fconc$. Now define $\fconcstargen$ to be the class of all $\varphi\in\fconc$ such that $\varphi(t)/\sqrt{t}$ is strictly decreasing, and  define $\fconcstardualgen$ to be the class of all $\varphi\in\fconc$ such that $\varphi(t)/\sqrt{t}$ is strictly increasing. The classes $\fconcstargen$ and $\fconcstardualgen$ are dual to one another in the following sense.
\begin{lemma}\label{dualconc}
$\varphi\in\fconcstargen$ if and only if $\varphi^*\in\fconcstardualgen$.
\end{lemma}

\begin{proof}
Let $\varphi\in\fconcstargen$,  $\varphi^*\in\fconcstardualgen$ 
and $t,u\in(0,\infty)$ with $t<u$. Also let $s=1/t$ and $v=1/u$ and note that $s>v$. Then 
\begin{align*}
\frac{\varphi(t)}{\sqrt{t}}\geq \frac{\varphi(u)}{\sqrt{u}}
&\Longleftrightarrow 
\frac{t\varphi^*(1/t)}{\sqrt{t}}
\geq \frac{u\varphi^*(1/u)}{\sqrt{u}}
\Longleftrightarrow
\sqrt{t}\varphi^*(1/t) \geq \sqrt{u}\varphi^*(1/u)
\Longleftrightarrow
\frac{\varphi^*(s)}{\sqrt{s}} \geq \frac{\varphi^*(v)}{\sqrt{v}}.
\end{align*}
\end{proof}

Note that if $\varphi(t)=\varphi(1)\sqrt{t}$,  then $\varphi(t)=\varphi^*(t)$ and for any $r>0$, $\as_\varphi(rB_n)=\varphi(1)|\partial B_n|$ is constant with respect to $r$. Combined with Lemma \ref{dualconc}, the next result shows that the ``self-dual" function $\varphi(t)=\varphi(1)\sqrt{t}$ serves as a transition where the $L_\varphi$ affine surface area reverses monotonicity on Euclidean balls. 

\begin{lemma}\label{monolem}
Let $\varphi\in\fconcstargen$, and let $r,s>0$ with $r\leq s$. Then
\begin{align*}
    \as_\varphi(rB_n) &\leq \as_\varphi(sB_n)\\
    \as_{\varphi^*}(rB_n) &\geq \as_{\varphi^*}(sB_n).
\end{align*}
\end{lemma}

\begin{proof}
Set $x=r^{-2n}$ and $y=s^{-2n}$. Then $x\geq y$, so by \eqref{asball} and Lemma \ref{dualconc}, 
\begin{align}\label{phistarmono}
\as_\varphi(rB_n)&=\frac{\varphi(x)}{\sqrt{x}}|\partial B_n| \leq \frac{\varphi(y)}{\sqrt{y}}|\partial B_n|=\as_\varphi(sB_n)\\
\as_{\varphi^*}(rB_n) &= \frac{\varphi^*(x)}{\sqrt{x}}|\partial B_n| \geq \frac{\varphi^*(y)}{\sqrt{y}}|\partial B_n|=\as_{\varphi^*}(sB_n). 
\end{align}
\end{proof}

\begin{remark}\label{fnexample1}
\normalfont The family of functions $\varphi_m:(0,\infty)\to(0,\infty)$ defined  by $\varphi_m(t):=\arctan(t^{1/m})$, $m\geq 2$, is a subclass of $\fconcstargen$; see Subsection \ref{arctanexamplepf} for the details.  
Hence by Lemma \ref{dualconc} the family $\varphi_m^*(t)=t\arctan(t^{-1/m})$, $m\geq 2$, is a subclass of $\fconcstardualgen$.  
Note that, for any $r$, the function $\varphi_m$ is not homogeneous of degree $r$. To see this, observe that if a function $f$ is homogeneous of degree $r$, then $f(t)=f(1)t^r$ and hence $f(t)f(1/t)=f(1)^2$ for all $t$. Thus if $\varphi_m$ is homogeneous for some $r$, then $\arctan(t^{1/m})\arctan(t^{-1/m})=\arctan(1)^2=\pi^2/16$  for all $t$, a contradiction. 

\end{remark}

Finally, we remark that the inclusions $\fconcstargen\subsetneq \fconc$ and $\fconcstardualgen\subsetneq\fconc$ are strict, as the example $\varphi(t)=\log(t+1)$ shows.

\vspace{2mm}

\section{Extremal general affine surface areas}\label{defsection}

For  a convex body $K$ in $\R^n$ with $g(K)=o$,  let
\begin{align*}
    \mathscr{C}_{\text{in}}(K)&:=\{ K^\prime\subset K: K^\prime\text{ is a convex body in }\R^n, \, g(K^\prime)=o\}\\
    \mathscr{C}_{\text{out}}(K)&:=\{K^\prime\supset K: K^\prime \text{ is a convex body in }\R^n, \, g(K^\prime)=o\},
\end{align*}
denote the families of all inscribed and circumscribed bodies with centroid at the origin, respectively. The condition $g(K^\prime)=o$ allows us to apply Lemmas \ref{isoperlemphi} and \ref{isoperlem2}. A  discussion on this centroid assumption  can be found in Subsection \ref{centroidsec}. For $\varphi\in\fconc$, we define the {\it extremal $L_\varphi$ affine surface areas} by
\begin{alignat*}{4}
& \ISphi(K) && := \sup_{K^\prime\in\Cin(K)}\as_\varphi(K^\prime), \quad && \iSphi(K) && :=\inf_{K^\prime\in\Cin(K)}\as_\varphi(K^\prime)\\
& \OSphi(K) && :=\sup_{K^\prime\in\Cout(K)}\as_\varphi(K^\prime), \quad && \oSphi(K) && :=\inf_{K^\prime\in\Cout(K)}\as_\varphi(K^\prime).
\end{alignat*}

\noindent Likewise, for $\psi\in\fconv$ we define the {\it extremal $L_\psi$ affine surface areas} by
\begin{alignat*}{4}
&\ISpsi(K) &&:= \sup_{K^\prime\in\Cin(K)}\as_\psi(K^\prime), \quad &&\iSpsi(K)&&:=\inf_{K^\prime\in\Cin(K)}\as_\psi(K^\prime)\\
&\OSpsi(K)&&:=\sup_{K^\prime\in\Cout(K)}\as_\psi(K^\prime), \quad &&\oSpsi(K)&&:=\inf_{K^\prime\in\Cout(K)}\as_\psi(K^\prime).
\end{alignat*}
These definitions were motivated by those in \cite{Giladietal}, where the  extremal $L_p$ affine surface areas were introduced; choosing a specific function $\varphi$ or $\psi$ defined in terms of $n$ and $p$, we recover their definitions (see Subsection \ref{Lpsection}). We will be especially interested in the {\it inner maximal $L_\varphi$ affine surface area} $\ISphi(K)$ and the {\it outer minimal $L_\psi$ affine surface area} $\oSpsi(K)$. 
In Lemma \ref{existence}, we  show that for any convex body $K\in\mathcal{K}_o^n$ with $g(K)=o$ and any $\varphi\in\fconcstargen$ and $\psi\in\fconvstargen$, there exist $K_{\max}^\varphi\in\Cin(K)$ and $K_{\min}^\psi\in\Cout(K)$ satisfying $\as_\varphi(K_{\max}^\varphi)=\ISphi(K)$ and $\as_\psi(K_{\min}^\psi)=\oSpsi(K)$. 

We begin by ruling out the trivial cases in the definitions above. Since the $L_\varphi$ and $L_\psi$ affine surface areas have a ``$0/\infty$ property" on polytopes, one might expect the same for some of the extremal affine surface areas.
Indeed, if $P\in\Cin(K)$ is a polytope, then $\as_\varphi(P)=0$ and $\as_\psi(P)=\infty$, so $\iSphi(K)=0$ and $\ISpsi(K)=\infty$. Similarly, if $Q\in\Cout(K)$ is a polytope, then  $\oSphi(K)=0$ and $\OSpsi(K)=\infty$. Moreover, by \eqref{asball2} and the definition of $\fconv$ we obtain 
\begin{align*}
\iSpsi(K) \leq \inf_{rB_n\in\Cin(K)}\as_\psi(rB_n) 
&=\inf_{rB_n\in\Cin(K)}r^n\psi(r^{-2n})|\partial B_n|=0.
\end{align*}

Let $K\in\K$. Then $K^\prime\in\Cin(K)$ if and only if $(K^\prime)^\circ\in\Cout(K^\circ)$. Thus by \eqref{duality1},
\begin{equation}\label{duality}
\text{IS}_{\varphi}(K)
=\sup_{K^\prime\in\Cin(K)}\as_{\varphi}(K^\prime)
=\sup_{(K^\prime)^\circ \in \Cout(K^\circ)}\as_{\varphi^*}((K^\prime)^\circ)=\OSphistar(K^\circ),
\end{equation}
provided the suprema exist. Note that if $\varphi\in\fconcstargen$ and $\lim_{t\to 0}\frac{\varphi(t)}{\sqrt{t}}=\infty$, then by Lemmas \ref{dualconc} and \ref{monolem},
\begin{align*}
    \text{IS}_{\varphi^*}(K) = \sup_{K^\prime\in\Cin(K)}\as_{\varphi^*}(K^\prime) &\geq \sup_{rB_n\in\Cin(K)}\as_{\varphi^*}(rB_n)
    =\sup_{rB_n\in\Cin(K)}\frac{\varphi^*(r^{-2n})}{r^{-n}}|\partial B_n|=\infty.
\end{align*}
Moreover, since
\[
\lim_{t\to\infty}\frac{\varphi^*(t)}{\sqrt{t}} = \lim_{t\to\infty}\sqrt{t}\varphi(1/t)=\lim_{s\to 0}\frac{\varphi(s)}{\sqrt{s}},
\]
by Lemma \ref{dualconc} we have $\varphi\in\fconcstargen$ and $\lim_{t\to 0}\frac{\varphi(t)}{\sqrt{t}}=\infty$ if and only if $\varphi^*\in\fconcstardualgen$ and $\lim_{t\to\infty}\frac{\varphi^*(t)}{\sqrt{t}}=\infty$. For any such function $\varphi$, we have $\OSphi(K)=\infty$.


 \subsection{Extremal $L_\psi^*$ affine surface areas}
 
 For $K\in\K$ and $\psi\in\fconv$, Ludwig \cite{Ludwig2010}  defined the $L_\psi^*$ affine surface area by $\as^*_\psi(K):=\as_\psi(K^\circ)$. When   $\psi(t)=t^{\frac{n}{n+p}}$ and $p<-n$, one recovers the $L_p$ affine surface area defined by Sch\"utt and Werner \cite{SW2004}. This leads to the definition of the extremal $L_\psi^*$ affine surface areas: 
 \begin{alignat*}{4}
& \ISpsistar(K) && := \sup_{K^\prime\in\Cin(K)}\as^*_\psi(K^\prime), \quad && \iSpsistar(K) && :=\inf_{K^\prime\in\Cin(K)}\as^*_\psi(K^\prime)\\
& \OSpsistar(K) && :=\sup_{K^\prime\in\Cout(K)}\as^*_\psi(K^\prime), \quad && \oSpsistar(K) && :=\inf_{K^\prime\in\Cout(K)}\as^*_\psi(K^\prime).
\end{alignat*}
Again, since $K^\prime\in\Cin(K)$ if and only if $(K^\prime)^\circ\in\Cout(K^\circ)$, we have  
\begin{align*}
    \ISpsistar(K)&=\OSpsi(K^\circ)=\infty\\
    \OSpsistar(K) &= \ISpsi(K^\circ)=\infty\\
    \oSpsistar(K)&=\iSpsi(K^\circ)=0
\end{align*}
and
\begin{equation}\label{dualLpsi}
    \iSpsistar(K) = \oSpsi(K^\circ).
\end{equation}
Thus, from the $L_\psi^*$ family we shall only consider the {\it inner minimal $L_\psi^*$ affine surface area} $\iSpsistar$.


\subsection{Summary}

We summarize these cases in the following tables.  As we will see later, the functionals $\ISphi$ and $\OSphistar$ are finite for all $\varphi\in\fconcstargen$,   while  $\oSpsi$ and $\iSpsistar$ are finite for all $\psi\in\fconvstargen$. 
First, let  $\varphi\in\fconcstargen$ be such that $\lim_{t\to 0}\frac{\varphi(t)}{\sqrt{t}}=\infty$. Then we have:
\begin{center}
    \begin{tabular}{c|c|c|c|c|c}
       $\ISphi(K)$ & $\text{IS}_{\varphi*}(K)$  & $\iSphi(K)$ & $\OSphi(K)$ & $\OSphistar(K)$ & $\oSphi(K)$ \\ \hline
       $<\infty$ & $\infty$ & $0$ & $\infty$ & $<\infty$ & $0$
    \end{tabular}
\end{center}
Among these functionals, we shall only be interested in $\ISphi$ and $\OSphistar$ for $\varphi\in\fconcstargen$. Meanwhile, for  $\psi\in\fconvstargen$ we have:
\begin{center}
    \begin{tabular}{c|c|c|c|c}
     $\ISpsi(K)$ & $\iSpsi(K)$ & $\OSpsi(K)$ & $\oSpsi(K)$ & $\iSpsistar(K)$\\ \hline
        $\infty$ & $0$ & $\infty$ & $<\infty$ & $<\infty$
    \end{tabular}
\end{center}
 \noindent Among these functionals, we shall only be interested in  $\oSpsi$ and $\iSpsistar$. 


\subsection{Existence of extremal bodies}\label{excontsection}

In our first result, we show that the extremal general affine surface areas exist for all  $\varphi\in\fconcstargen$  and all  $\psi\in\fconvstargen$.  An analogous result was proved for the extremal $L_p$ affine surface areas in \cite[Lem. 3.2]{Giladietal}. 

\begin{lemma}\label{existence}
Let $K$ be a convex body in $\R^n$ with centroid at the origin, and let $\varphi\in\fconcstargen$ and $\psi\in\fconvstargen$. 
\begin{itemize}
    \item[(i)] There exists  $K_{\max}^\varphi\in\Cin(K)$ such that $\ISphi(K)=\as_\varphi(K_{\max}^\varphi)$.
    
    \vspace{1mm}
    
    \item[(ii)] There exists $K_{\max}^{\varphi^*}\in\Cout(K)$ such that $\OSphistar(K)=\as_{\varphi^*}(K_{\max}^{\varphi^*})$.
    
        \vspace{1mm}

    \item[(iii)] There exists  $K_{\min}^\psi\in\Cout(K)$ such that $\oSpsi(K)=\as_\psi(K_{\min}^\psi)$.
    
        \vspace{1mm}

     \item[(iv)] There exists  $K_{\min}^{\psi^*}\in\Cin(K)$ such that $\iSpsistar(K)=\as_\psi(K_{\min}^{\psi^*})$.
\end{itemize}
\end{lemma}
\noindent We postpone the proof until Section \ref{proofs}. 


\subsection{Monotonicity and Continuity}

Although the $L_\varphi$ and $L_\psi$ affine surface areas are not monotone in general, the corresponding extremal affine surface areas are. If $K,L\in\K$ and $K\subset L$, 
then $\Cin(K)\subset\Cin(L)$ implies
\begin{align}
\ISphi(K)&=\sup_{K^\prime\in\Cin(K)}\as_\varphi(K^\prime) \leq \sup_{L^\prime\in\Cin(L)}\as_\varphi(L^\prime) =\ISphi(L), \label{mono1}\\
\iSpsistar(K) &= \inf_{K^\prime\in\Cin(K)}\as^*_\psi(K^\prime) \geq \inf_{L^\prime\in\Cin(L)}\as^*_\psi(L^\prime)=\iSpsistar(L)\label{mono4},
\end{align}
while $\Cout(K)\supset\Cout(L)$ implies
\begin{align}
\OSphistar(K) &=\sup_{K^\prime\in\Cout(K)}\as_{\varphi^*}(K^\prime) \geq \sup_{L^\prime\in\Cout(L)}\as_{\varphi^*}(L^\prime)=\OSphistar(L) \label{mono3}\\
\oSpsi(K) &= \inf_{K^\prime\in\Cout(K)}\as_\psi(K^\prime) \leq \inf_{L^\prime\in\Cout(L)}\as_\psi(L^\prime)=\oSpsi(L)\label{mono2}.
\end{align}

The next result shows that the extremal general affine surface areas are continuous with respect to the Hausdorff metric (c.f. \cite[Prop. 3.3]{Giladietal}). This is not the case for the general $L_\varphi$ and $L_\psi$ affine surface areas, which are  zero or infinite, respectively, on any sequence of  polytopes that converges to a given convex body in the Hausdorff metric.

\begin{proposition}\label{contlem}
Consider the class of convex bodies in $\R^n$ with centroid at the origin, equipped with the Hausdorff metric. \begin{itemize}
    \item[(i)] For any  $\varphi\in\fconcstargen$, the functional $K\mapsto\ISphi(K)$ is continuous.
    \item[(ii)] For any  $\varphi^*\in\fconcstardualgen$, the functional $K\mapsto\OSphistar(K)$ is continuous.
    \item[(iii)] For any  $\psi\in\fconvstar$, the functional $K\mapsto\oSpsi(K)$ is continuous.
    \item[(iv)] For any  $\psi\in\fconvstar$, the functional $K\mapsto\iSpsistar(K)$ is continuous.
    \end{itemize}
\end{proposition}
\vspace{1mm}

\noindent The proof of Proposition \ref{contlem} is also given in Section \ref{proofs}.

\section{Lemmas}\label{lemmas}

The next lemma is elementary. We include the proof for the reader's convenience.
\begin{lemma}\label{ballLem}
 For all $\varphi\in\fconcstargen, \varphi^*\in\fconcstardualgen$ and   $\psi\in\fconvstargen$, we have: 
\begin{align}
    \ISphi(B_n) &= 
     \as_\varphi(B_n)=\varphi(1)|\partial B_n|\label{ballcase}\\
     \OSphistar(B_n) &= \as_{\varphi^*}(B_n)=\varphi(1)|\partial B_n|\label{ballcase3}\\
    \oSpsi(B_n)&=
    \as_\psi(B_n)=\psi(1)|\partial B_n| \label{ballcase2}\\
       \iSpsistar(B_n)&=
    \as^*_\psi(B_n)=\psi(1)|\partial B_n|. \label{ballcase4}
\end{align}
\end{lemma}

\begin{proof}
On one hand, we have $\ISphi(B_n)=\sup_{K\in\Cin(B_n)}\as_\varphi(K) \geq \as_\varphi(B_n)$ 
since $B_n\in\Cin(B_n)$. On the other hand, by  Lemmas \ref{isoperlemphi} and   \ref{monolem} we derive
\begin{align*}
    \ISphi(B_n) =\sup_{K\in\Cin(B_n)}\as_\varphi(K)
    &\leq\sup_{K\in\Cin(B_n)}\as_\varphi(\vrad(K)B_n)
    = \as_\varphi(B_n).
\end{align*}
Thus $\ISphi(B_n)=\as_\varphi(B_n)$. By \eqref{duality}, we have 
    $\OSphistar(B_n) =\ISphi(B_n^\circ)=\ISphi(B_n)$ 
since $B_n^\circ=B_n$. The argument to prove $\oSpsi(B_n)=\as_\psi(B_n)$  is similar, and by \eqref{dualLpsi} we have  $\iSpsistar(B_n)=\as_\psi^*(B_n)$. 

The identity  $\as_\varphi(B_n)=\varphi(1)|\partial B_n|$ follows from \eqref{asball}, and  $\as_\psi(B_n)=\psi(1)|\partial B_n|$ follows from \eqref{asball2}. Finally, note that $\as_{\varphi}(B_n)=\as_{\varphi^*}(B_n)$ and $\as_\psi^*(B_n)=\as_\psi(B_n)$.
\end{proof}

Another ingredient we will need describes the effect a dilation has on the general affine surface area of a convex body.

\begin{lemma}\label{scalinglem}
Let $K\in\mathcal{K}_o^n$, $\varphi\in\fconc$ and $\psi\in\fconv$. Then for any $r\in[0,1]$ and any $R\geq 1$, the following inequalities hold:
\begin{itemize}
    \item[(i)] $\as_\varphi(rK) \geq r^n \as_\varphi(K)$
    
    \item[(ii)] $\as_\varphi(RK) \leq R^n\as_\varphi(K)$
    
    \item[(iii)] $\as_\psi(rK) \leq r^n\as_\psi(K)$
    
    \item[(iv)] $\as_\psi(RK) \geq R^n\as_\psi(K)$.
\end{itemize}
\end{lemma}

\begin{proof}
We shall prove part (iv);  the other cases (i)-(iii) proceed similarly. By the definition \eqref{defaspsi} of the $L_\psi$ affine surface area,
\begin{align*}
\as_\psi(rK)  &=\int_{\partial rK} \psi\left(\frac{\kappa(rK,x)}{\langle x,N_{\partial rK}(x)\rangle^{n+1}}\right)\langle x,N_{\partial rK}(x)\rangle\,d\mu_{\partial rK}(x)\\
&=\int_{\partial rK} \psi\left(\frac{\kappa(rK,ry)}{\langle ry,N_{\partial rK}(ry)\rangle^{n+1}}\right)\langle ry,N_{\partial rK}(ry)\rangle\,d\mu_{\partial rK}(ry)\\
&=r^n\int_{\partial K} \psi\left(\frac{\kappa(K,y)}{r^{2n}\langle y,N_{\partial K}(y)\rangle^{n+1}}\right)\langle y,N_{\partial K}(y)\rangle\,d\mu_{\partial K}(y)\\
&\leq r^n\int_{\partial K} \psi\left(\frac{\kappa(K,y)}{\langle y,N_{\partial K}(y)\rangle^{n+1}}\right)\langle y,N_{\partial K}(y)\rangle\,d\mu_{\partial K}(y)\\
&=r^n\as_\psi(K).
\end{align*}
In the  inequality we used the fact that $r\leq 1$ and $\psi$ is decreasing.
\end{proof}

\section{Affine isoperimetric inequalities}\label{isoperimetry}

Affine isoperimetric inequalities are among the most powerful tools in convex geometry. They relate two functionals on the class of convex bodies in $\R^n$, where the ratio of the functionals is invariant under nondegenerate linear  transformations. Prominent examples include the classical affine isoperimetric inequality \cite{Leichtweiss, Lutwak91, Hug96} and the Blaschke-Santal\'o inequality \cite{Blaschke1917, Petty1985, Santalo1949}. More recent examples include  affine isoperimetric inequalities for the $L_p$ affine surface area \cite{Lutwak96, MeyerWerner2000, WernerYe08}, the mixed $L_p$ affine surface area \cite{Lutwak96,WG2007,WernerYe2010}, the general affine surface area \cite{Ludwig2010, Deping2014Steiner}, the mixed general affine surface area \cite{Deping2012} and the extremal $L_p$ affine surface areas \cite{Giladietal}. In our next two results, we provide  affine isoperimetric inequalities for the extremal general affine surface areas.

\begin{proposition}\label{isoperimetric} 
Let $K$ be a convex body in $\R^n$ with centroid at the origin. If $\varphi\in\fconcstargen$, then
\begin{align}
\ISphi(K) &\leq  \vrad(K)^n\cdot\frac{\varphi(\vrad(K)^{-2n})}{\varphi(1)}\cdot \ISphi(B_n) \label{ISisop}\\
\OSphistar(K) &\leq \vrad(K^\circ)^{-n}\cdot\frac{\varphi^*(\vrad(K^\circ)^{2n})}{\varphi^*(1)}\cdot \OSphistar(B_n). \label{OSisop} 
\end{align}
Moreover, if $\varphi$ is strictly increasing, then equality holds in each case if and only if $K$ is an ellipsoid.
\end{proposition}

\begin{proposition}\label{isoperimetric2} 
Let $K$ be a convex body in $\R^n$ with centroid at the origin. If $\psi\in\fconvstargen$, then
\begin{align}
    \oSpsi(K)&\geq \vrad(K)^n\cdot\frac{\psi(\vrad(K)^{-2n})}{\psi(1)}\cdot \oSpsi(B_n)\label{osisop}\\
      \iSpsistar(K)&\geq \vrad(K^\circ)^n\cdot\frac{\psi(\vrad(K^\circ)^{-2n})}{\psi(1)}\cdot \iSpsistar(B_n)\label{isstarisop}.
\end{align}
Moreover, if $\psi$ is strictly decreasing, then equality holds in each case if and only if $K$ is an ellipsoid.
\end{proposition}

\subsection{Proof of Propositions  \ref{isoperimetric} and \ref{isoperimetric2}}
By Lemmas \ref{isoperlemphi}, \ref{monolem} and \ref{ballLem} we obtain 
\begin{align*}
    \ISphi(K) = \sup_{K^\prime\in\Cin(K)}\as_\varphi(K^\prime)
    &\leq \sup_{K^\prime\in\Cin(K)}\as_\varphi(\vrad(K^\prime)B_n)\\
    &= \as_\varphi(\vrad(K)B_n)\\
    &=\vrad(K)^n\cdot\frac{\varphi(\vrad(K)^{-2n})}{\varphi(1)}\cdot\varphi(1)|\partial B_n|\\
    &=\vrad(K)^n\cdot\frac{\varphi(\vrad(K)^{-2n})}{\varphi(1)}\cdot\ISphi(B_n).
\end{align*}
Next, we  use \eqref{ISisop}, \eqref{duality} and $\varphi^*(t)=t\varphi(1/t)$ to get
\begin{align*}
\OSphistar(K)=\ISphi(K^\circ) &\leq \vrad(K^\circ)^n\cdot\frac{\varphi(\vrad(K^\circ)^{-2n})}{\varphi(1)}\cdot \OSphistar(B_n)\\
&=\vrad(K^\circ)^{-n}\cdot\frac{\varphi^*(\vrad(K^\circ)^{2n})}{\varphi^*(1)}\cdot \OSphistar(B_n).
\end{align*}

Similarly, using  Lemma \ref{isoperlem2}, Lemma \ref{ballLem} and \eqref{monotonepsi} we derive  
\begin{align*}
    \oSpsi(K) = \inf_{K^\prime\in\Cout(K)}\as_\psi(K^\prime)
    &\geq\as_\psi(\vrad(K)B_n)\\
    &=\vrad(K)^n\cdot\frac{\psi(\vrad(K)^{-2n})}{\psi(1)}\cdot\oSpsi(B_n).
\end{align*}
From \eqref{osisop} and \eqref{dualLpsi} we get \eqref{isstarisop}. The equality conditions in each case follow from those in Lemmas \ref{isoperlemphi} and \ref{isoperlem2}. 
\qed

\begin{remark}
\normalfont Choosing $\varphi(t)=t^{\frac{p}{n+p}}$ with $p\in[0,n]$ or $p\in[n,\infty]$ in Proposition \ref{isoperimetric}, or $\psi(t)=t^{\frac{p}{n+p}}$ with $p\in(-n,0]$ in  Proposition \ref{isoperimetric2}, we recover the affine isoperimetric inequalities of  Giladi, Huang, Sch\"utt and Werner  \cite[Prop. 3.4]{Giladietal} for the extremal $L_p$ affine surface areas.
\end{remark}
\subsection{Blaschke-Santal\'o type  inequalities}\label{BSsect}

The celebrated Blaschke-Santal\'o inequality states that for any convex body $K$ in $\R^n$ with centroid at the origin,
\begin{equation*}
    |K|\cdot|K^\circ| \leq |B_n|^2
\end{equation*}
with equality if and only if $K$ is an  ellipsoid. For $n=2,3$, this result is due to Blaschke \cite{Blaschke1917}, and it was  extended to all  $n$ by Santal\'o \cite{Santalo1949}. The equality conditions were later proved by Petty \cite{Petty1985}. In particular, Blaschke-Santal\'o type inequalities have also been shown for the $L_p$ affine surface area and the mixed $L_p$ affine surface area (see \cite{Lutwak96, WG2007,  WernerYe08, WernerYe2010, Deping2012} and the references therein), as well as to the general affine surface area \cite{Ludwig2010} and mixed general affine surface area \cite{Deping2012}. In our main result, we prove an analogue for the extremal general affine surface areas. 

\begin{theorem}\label{BSvarphi}
Let $K$ be a convex body in $\R^n$ with centroid at the origin, and let  $\varphi\in\fconcstargen$ satisfy the submultiplicativity condition $\varphi(t)\varphi(1/t)\leq\varphi(1)^2$ for all $t>0$. Then
\begin{align}
    \ISphi(K)\ISphi(K^\circ)
    &\leq \ISphi(B_n)^2\label{BS1}\\ 
    \OSphistar(K)\OSphistar(K^\circ) &\leq \OSphistar(B_n)^2\label{BS2}. 
\end{align}
Moreover, if $\varphi$ is strictly increasing, then equality holds in each case if and only if $K$ is an ellipsoid.
\end{theorem}

\begin{proof}
Applying  Proposition \ref{isoperimetric} to $K$ and $K^\circ$, we derive
\begin{align}\label{submultBS}
    \ISphi(K)\ISphi(K^\circ)
    &\leq  \frac{\varphi(\vrad(K)^{-2n})}{\vrad(K)^{-n}}\cdot\frac{\varphi(\vrad(K^\circ)^{-2n})}{\vrad(K^\circ)^{-n}}\cdot\varphi(1)^{-2}\cdot\ISphi(B_n)^2.
\end{align}
It suffices to show that for all $s,t>0$ with $st\geq 1$, 
\begin{equation}\label{submulteqn}
    \frac{\varphi(s)}{\sqrt{s}}\cdot\frac{\varphi(t)}{\sqrt{t}}\leq \varphi(1)^2,
\end{equation}
for then the conclusion follows with $s=\vrad(K)^{-2n}$ and $t=\vrad(K^\circ)^{-2n}$, where  $st\geq 1$ holds by the Blaschke-Santal\'o inequality. The function $t\mapsto \varphi(t)/\sqrt{t}$ is decreasing since $\varphi\in\fconcstargen$, so for all $s,t>0$ with $s\geq 1/t$, 
\[
\frac{\varphi(s)}{\sqrt{s}}\cdot\frac{\varphi(t)}{\sqrt{t}}\leq \frac{\varphi(1/t)}{\sqrt{1/t}}\cdot\frac{\varphi(t)}{\sqrt{t}}=\varphi(t)\varphi(1/t).
\]
By hypothesis, this is less  than or equal to  $\varphi(1)^2$ for all $t>0$, which shows \eqref{submulteqn}. Hence, inequality \eqref{BS1} follows from \eqref{submultBS}.

 The identity $ \ISphi(K)\ISphi(K^\circ)=\OSphistar(K)\OSphistar(K^\circ)$ follows from \eqref{duality}, and $\ISphi(B_n)=\OSphistar(B_n)$ by Lemma \ref{ballLem}. Thus, \eqref{BS2} follows from \eqref{BS1}. The equality conditions follow from those in Proposition \ref{isoperimetric}, since they are stronger than those of the Blaschke-Santal\'o inequality.
\end{proof}

\begin{remark}\label{fnexample2}
\normalfont For any $m\geq 2$, the function $\varphi_m(t)=\arctan(t^{1/m})$ lies in $\fconcstargen$ by Remark \ref{fnexample1}. We show that it satisfies the submultiplicativity condition $\varphi_m(t)\varphi_m(1/t)\leq \varphi_m(1)^2$ for all $t>0$. This is equivalent to 
\begin{equation}\label{toshow}
    \arctan(t^{1/m})\arctan(t^{-1/m})\leq \frac{\pi^2}{16}, \quad \forall t>0.
\end{equation}
We use the identity $\arctan(1/x)=\pi/2-\arctan(x)$ for all $x>0$ to derive
\[
\varphi_m(t)\varphi_m(1/t)=\varphi_m(t)\left(\frac{\pi}{2}-\varphi_m(t)\right)=-\varphi_m(t)^2+\frac{\pi}{2}\cdot\varphi_m(t).
\]
Set $s=\varphi_m(t)$. The maximum value of the function $s\mapsto -s^2+\frac{\pi}{2}s$ is achieved when $s=\pi/4$, which implies $t=\tan(\pi/4)=1$. Therefore,
\[
\varphi_m(t)\varphi_m(1/t)\leq -\varphi_m(1)^2+\frac{\pi}{2}\varphi_m(1)=\frac{\pi^2}{16},
\]
which proves \eqref{toshow}. Thus, for $m\geq 2$ the function $\varphi_m$ satisfies the hypotheses of Theorem \ref{BSvarphi}.
\end{remark}



\subsection{Inverse Santal\'o type inequalities}\label{invsantsec}

Bourgain and Milman's inverse Santal\'o inequality \cite{BourgainMilman} (see also \cite{Kuperberg08, MilmanPajor2000, Nazarov2012}) states that there exists an absolute constant $c>0$ such that for all $n\geq 1$ and every  convex body $K$ in $\R^n$, 
\begin{equation}\label{reverseBS}
    |K|\cdot |K^\circ| \geq c^n |B_n|^2.
\end{equation}
The best known constant is $c=1/2$, which is due to Kuperberg \cite{Kuperberg08}. 

\vspace{2mm}

Our next result gives inverse Santal\'o type inequalities for the outer minimal $L_\psi$ and inner minimal $L_\psi^*$ affine surface areas. 

\begin{theorem}\label{reverseBSospsi}
Let $K$ be a convex body in $\R^n$ with centroid at the origin. If $\psi\in\fconvstargen$, then there exists a positive absolute constant $c$ such that 
\begin{align}
\oSpsi(K)\oSpsi(K^\circ) &\geq c^n\cdot\frac{\psi(\vrad(K)^{-2n})\psi(\vrad(K^\circ)^{-2n})}{\psi(1)^2}\cdot\oSpsi(B_n)^2\label{reverseBS1}\\
\iSpsistar(K)\iSpsistar(K^\circ) &\geq c^n\cdot\frac{\psi(\vrad(K)^{-2n})\psi(\vrad(K^\circ)^{-2n})}{\psi(1)^2}\cdot\iSpsistar(B_n)^2.\label{reverseBS2}
\end{align}
\end{theorem}

\begin{proof}
By   Proposition \ref{isoperimetric2} and the inverse Santal\'o inequality \eqref{reverseBS},
\begin{align*}
    \oSpsi(K)\oSpsi(K^\circ)&
      \geq (\vrad(K)\vrad(K^\circ))^n\psi(\vrad(K)^{-2n})\psi(\vrad(K^{\circ})^{-2n})\oSpsi(B_n)^2\\
      &\geq c^n\cdot\frac{\psi(\vrad(K)^{-2n})\psi(\vrad(K^{\circ})^{-2n})}{\psi(1)^2}\cdot\oSpsi(B_n)^2.
\end{align*}
Inequality  \eqref{reverseBS2} now follows from \eqref{dualLpsi} and \eqref{ballcase4}. 
\end{proof}



\subsection{Extremal $L_p$ affine surface areas}\label{Lpsection}

Let $K$ be a convex body in $\R^n$ with centroid at the origin. For any real number $p\neq -n$,  the $L_p$ {\it affine surface area} $\as_p(K)$ of $K$ is defined by
\begin{equation}\label{asp}
\as_p(K)=\int_{\partial K}\frac{\kappa(K,x)^{\frac{p}{n+p}}}{\langle x, N_{K}(x)\rangle^{\frac{n(p-1)}{n+p}}}\,d\mu_{\partial K}(x)
\end{equation}
and \cite{Lutwak96,SW2004}
\begin{equation}\label{asinfty}
    \as_{\pm\infty}(K)=\int_{\partial K}\frac{\kappa(K,x)}{\langle x,N_K(x)\rangle^n}\,d\mu_{\partial K}(x),
\end{equation}
provided the integrals exist. These definitions were given  in \cite{Lutwak96} for $p>1$ and \cite{SW2004} for $p<1$. The case $p=1$ gives the classical affine surface $\as_1$ from affine differential geometry, originally due to Blaschke  \cite{Blaschke1923} for sufficiently smooth convex bodies. The definition of $\as_1$ was extended to all convex bodies in $\R^n$ by several authors  \cite{Leichtweiss1986, Lutwak91, MeyerWerner2000, Schmuck1992, SW1990, Werner1994}, with Sch\"utt and Werner \cite{SW1990} showing specifically that the definition $\as_1(K)=\int_{\partial K}\kappa(K,x)^{\frac{1}{n+1}}\,d\mu_{\partial K}(x)$  extends naturally to all convex bodies in $\R^n$. 

The $L_p$ affine surface area is homogeneous of degree $\frac{n(n-p)}{n+p}$ (see \cite[Prop. 9]{SW2004}), meaning
\[
\as_p(\lambda K)=\lambda^{\frac{n(n-p)}{n+p}}\as_p(B_n),\quad \forall \lambda>0.
\]
The following $L_p$ affine isoperimetric inequalities are due to Lutwak \cite{Lutwak96} for $p\geq 1$ and to Werner and Ye \cite[Thm. 4.2]{WernerYe08} for  $p<1$. For any convex body $K$ in $\R^n$ with centroid at the origin and all $p\geq 0$,
\begin{equation}\label{Lp1}
    \as_p(K)\leq \vrad(K)^{\frac{n(n-p)}{n+p}}\as_p(B_n)
\end{equation}
while for $-n<p\leq 0$ the inequality reverses,
\begin{equation}\label{Lp2}
    \as_p(K) \geq \vrad(K)^{\frac{n(n-p)}{n+p}}\as_p(B_n).
\end{equation}
Equality holds in each case if and only if $K$ is an ellipsoid, and it holds trivially if $p=0$. For $p=\pm\infty$, it follows from  \cite[Thm. 7.7]{Lutwak96} and  \cite[pp. 114--115]{SW2004} that for any convex body $K$ in $\R^n$, 
\begin{equation}\label{Linfinity}
\as_{\pm\infty}(K) \leq n|K^\circ|
\end{equation}
with equality if and only if $K$ is an ellipsoid.

Giladi, Huang, Sch\"utt and Werner \cite{Giladietal} defined the {\it inner and outer maximal $L_p$ affine surface areas} by
\[
\text{IS}_p(K) = \sup_{K^\prime\in\Cin(K)}\as_p(K^\prime), \quad \text{OS}_p(K) = \sup_{K^\prime\in\Cout(K)}\as_p(K^\prime),
\]
respectively, and the {\it inner and outer minimal $L_p$ affine surface areas} by
\[
\text{is}_p(K) = \inf_{K^\prime\in\Cin(K)}\as_p(K^\prime), \quad \text{os}_p(K) = \inf_{K^\prime\in\Cout(K)}\as_p(K^\prime),
\]
respectively. The relevant $p$ ranges for $\text{IS}_p, \text{OS}_p$ and $\text{os}_p$ are the intervals $[0,n]$, $[n,\infty]$ and $(-n,0]$ respectively; there is no interesting $p$ range for the functional $\text{is}_p$, which is identically zero for all $K$ and all $p$.  Taking $\varphi(t)=t^{\frac{p}{n+p}}$ or $\psi(t)=t^{\frac{p}{n+p}}$ with the corresponding $p$ interval, we recover these definitions from those of the extremal $L_\varphi$ and $L_\psi$ affine surface areas given in Section \ref{defsection}. 

\vspace{2mm}

As a corollary to Theorem \ref{BSvarphi}, we obtain  Blaschke-Santal\'o type inequalities for the extremal $L_p$ affine surface areas.

\begin{corollary}\label{BSineqs}
Let $K$ be a convex body in $\R^n$ with centroid at the origin.
\begin{itemize}
\item[(i)] For all $p\in[0,n]$ we have  $\ISp(K)\ISp(K^\circ) \leq \ISp(B_n)^2$. Equality holds trivially if $p=n$.

\vspace{1mm}

\item[(ii)] For all $p\in[n,\infty]$ we have  
    $\OSp(K)\OSp(K^\circ) \leq \OSp(B_n)^2$. Equality holds trivially if $p=n$.
\end{itemize} 
Equality holds in each case  if and only if $K$ is an ellipsoid.
\end{corollary}

\begin{proof}
Let $\varphi(t)=t^{\frac{p}{n+p}}$. For $p\in[0,n]$, we have $\frac{p}{n+p}\in[0,1/2]$. Thus, $\lim_{t\to 0}\varphi(t)=\lim_{t\to\infty}\frac{\varphi(t)}{t}=0$ and $\varphi(0)=0$. For $p\in[0,n)$, the function  $\varphi(t)/\sqrt{t}=t^{\frac{p-n}{2(n+p)}}$ is strictly decreasing. Since $\varphi^{\prime\prime}(t)<0$ for $p\in(0,n)$, this shows that  $\varphi\in\fconcstargen$ for all $p\in(0,n)$. Moreover, $\varphi$ is  strictly increasing when $p\in(0,n]$, and $\varphi(t)\varphi(1/t)=\varphi(1)^2$. Thus $\varphi$ satisfies the hypotheses of Theorem \ref{BSvarphi} for any $p\in(0,n)$, so $\ISp(K)\ISp(K^\circ)\leq\ISp(B_n)^2$ holds for all $p\in(0,n)$.

For $p=0$, it was shown in \cite{Giladietal} that $\text{IS}_0(K)=n|K|$. Thus by the Blaschke-Santal\'o inequality,
\[
\text{IS}_0(K)\text{IS}_0(K^\circ) = n^2|K|\cdot|K^\circ|\leq n^2| B_n|^2=\text{IS}_0(B_n)^2.
\]
For $p=n$, it was also shown in \cite{Giladietal} that $\text{IS}_n(K)=|\partial B_n|$, so
\[
\text{IS}_n(K)\text{IS}_n(K^\circ)=|\partial B_n|^2=\text{IS}_n(B_n)^2
\]
and equality holds trivially. This proves (i). 

For (ii), note that $\varphi^*(t)=t^{\frac{n}{n+p}}$. From  \eqref{genas1} and \eqref{asp} it follows that (see also \cite{Hug96-2}, \cite[Cor. 3.1]{WernerYe08} and  \cite[Thm. 4]{Ludwig2010})
 \[
 \as_{\varphi^*}(K)=\as_{n^2/p}(K).
 \]
 Therefore, by \eqref{duality},  the Blaschke-Santal\'o inequality for $\ISp$ and Lemma \ref{ballLem},
 \[
 \text{OS}_{n^2/p}(K)\text{OS}_{n^2/p}(K^\circ)=\ISp(K)\ISp(K^\circ)\leq \ISp(B_n)^2= \text{OS}_{n^2/p}(B_n)^2.
 \]
 Now  we  replace $n^2/p$ for $p\in(0,n]$ by $p$ for $p\in[n,\infty)$ to obtain the desired inequality. In the case $p=n$, equality holds trivially since
 \[
 \text{OS}_n(K)\text{OS}_n(K^\circ)=\text{IS}_n(K)\text{IS}_n(K^\circ)=\text{IS}_n(B_n)^2.
 \]
 
 For the special case $p=\infty$, we apply \eqref{Linfinity} and the Blaschke-Santal\'o inequality to get 
 \begin{align*}
 \text{OS}_\infty(K)\text{OS}_\infty(K^\circ) 
 &= \sup_{K^\prime\in\Cout(K)}\as_\infty(K^\prime)\cdot\sup_{K^{\prime\prime}\in\Cout(K^\circ)}\as_\infty(K^{\prime\prime})\\
 &\leq \sup_{K^\prime\in\Cout(K)}n|(K^\prime)^\circ|\cdot\sup_{K^{\prime\prime}\in\Cout(K^\circ)}n|(K^{\prime\prime})^\circ|\\
 &=\sup_{(K^\prime)^\circ\in\Cin(K^\circ)}n|(K^\prime)^\circ|\cdot\sup_{(K^{\prime\prime})^\circ\in\Cin(K)}n|(K^{\prime\prime})^\circ|\\
 &\leq n^2|K|\cdot|K^\circ|\\
 &\leq n^2|B_n|^2\\
 &=\text{OS}_\infty(B_n)^2.
 \end{align*}


The equality conditions in parts (i) and (ii) follow from those in Theorem \ref{BSvarphi}.
\end{proof}

We also obtain the following inverse Santal\'o  type inequalities for the extremal $L_p$ affine surface areas. To state the result, we first define the following notation. For $p<-n$ and $\psi(t)=t^{\frac{n}{n+p}}$, let  $\as_p^*(K):=\as_\psi^*(K)=\as_\psi(K^\circ)=\as_{n^2/p}(K^\circ)$ and define  $\iSpstar(K):=\inf_{K^\prime\in\Cin(K)}\as_p^*(K^\prime)$.
\begin{theorem}
Let $K$ be a convex body in $\R^n$ with centroid at the origin.
\begin{itemize}
    \item[(i)] Let  $p\in(-n,0]$. There exists a positive absolute constant $c$ such that $\oSp(K)\oSp(K^\circ) \geq c^{\frac{n(n-p)}{n+p}}\oSp(B_n)^2$.
    
    \item[(ii)] Let  $p\in(-\infty,-n)$. There exists a positive absolute constant $c$ such that $\iSpstar(K)\iSpstar(K^\circ) \geq c^{\frac{n(n-p)}{n+p}}\iSpstar(B_n)^2$.
\end{itemize}
\end{theorem}

\begin{proof}
Using  inequalities \eqref{Lp2} and \eqref{reverseBS}  we derive
\begin{align*}
    \oSp(K)\oSp(K^\circ) &=\inf_{K^\in\Cout(K)}\as_p(K^\prime)\cdot\inf_{K^{\prime\prime}\in\Cout(K^\circ)}\as_p(K^{\prime\prime})\\
    &\geq\inf_{K^\prime\in\Cout(K)}\vrad(K^\prime)^{\frac{n(n-p)}{n+p}}\as_p(B_n)\cdot\inf_{K^{\prime\prime}\in\Cout(K^\circ)}\vrad(K^{\prime\prime})^{\frac{n(n-p)}{n+p}}\as_p(B_n)\\
    &=(\vrad(K)\vrad(K^\circ))^{\frac{n(n-p)}{n+p}}\as_p(B_n)^2\\
    &\geq c^{\frac{n(n-p)}{n+p}}\as_p(B_n)^2\\
    &=c^{\frac{n(n-p)}{n+p}}\oSp(B_n)^2.
\end{align*}
Part (ii) now follows from (i) by using     $\iSpstar(K)=\text{os}_{n^2/p}(K^\circ)$ for $p\in(-\infty,-n)$.
\end{proof}

\section{Appendix}\label{proofs}

\subsection{Proof of Lemma \ref{existence}}

The proof  is similar to that of  \cite[Lemma 3.2]{Giladietal}; we include the arguments for the reader's convenience. Let $\varphi\in\fconcstargen$. By  Lemmas \ref{isoperlemphi} and \ref{monolem},
\begin{align*}
    \ISphi(K) =\sup_{K^\prime\in\Cin(K)}\as_\varphi(K^\prime)
     &\leq\sup_{K^\prime\in\Cin(K)}\as_\varphi(\vrad(K^\prime)B_n)\\
    &=\sup_{K^\prime\in\Cin(K)}\vrad(K^\prime)^n\varphi(\vrad(K^\prime)^{-2n})|\partial B_n|\\
   &\leq \vrad(K)^n\varphi(\vrad(K)^{-2n})|\partial B_n|\\
   &=n|K|\varphi(\vrad(K)^{-2n}),
\end{align*} 
which is finite.  Hence, there exists a sequence $\{C_k\}_{k\in\mathbb{N}}\subset\Cin(K)$ such that for all $k\in\mathbb{N}$, 
\[
\as_\varphi(C_k)+\frac{1}{k} \geq \sup_{K^\prime\in\Cin(K)}\as_\varphi(K^\prime).
\] 
On the other hand, $\as_\varphi(C_k)\leq \sup_{K^\prime\in\Cin(K)}\as_\varphi(K^\prime)$ since $C_k\in\Cin(K)$, so by the squeeze theorem
\[
\lim_{k\to\infty}\as_\varphi(C_k) = \sup_{K^\prime\in\Cin(K)}\as_\varphi(K^\prime).
\]
By the Blaschke selection theorem, there exists a subsequence $\{C_{k_j}\}_{j\in\mathbb{N}}$ that converges to a convex set $K_0\subset K$ with respect to the Hausdorff metric as $j\to\infty$. We claim that $K_0$ has nonempty interior. To see this, suppose not. Then 
$\lim_{j\to\infty}|C_{k_j}|=|K_0|=0$.  By  Lemma \ref{isoperlemphi} and the definition of $\fconcstargen$ this implies
\[
\lim_{j\to\infty}\as_\varphi(C_{k_j})
\leq \lim_{j\to\infty}\as_\varphi(\vrad(C_{k_j})B_n)
=\lim_{j\to\infty}\vrad(C_{k_j})^n\varphi(\vrad(C_{k_j})^{-2n})|\partial B_n|=0.
\]
Since $K$ contains the origin in its interior, there exists $\varepsilon>0$ such that $\varepsilon B_n\subset K$. 
Thus,
\[
0=\lim_{j\to\infty}\as_\varphi(C_{k_j})=\sup_{K^\prime\in\Cin(K)}\as_\varphi(K^\prime)\geq \as_\varphi(\varepsilon B_n)=\varepsilon^n\varphi(\varepsilon^{-2n})|\partial B_n|>0,
\]
a contradiction. We have therefore shown that $K_0$ is a convex body in $\R^n$. 

Finally, we show that $\ISphi(K)=\as_\varphi(K_0)$. On one hand,  $\ISphi(K) \geq \as_\varphi(K_0)$ since $K_0\in\Cin(K)$, while on the other hand we can apply the upper semicontinuity \eqref{uppersemi} of the $L_\varphi$ affine surface area to get 
\[
\ISphi(K)=\sup_{K^\prime\in\Cin(K)}\as_\varphi(K^\prime)=\limsup_{j\to\infty}\as_\varphi(C_{k_j})\leq \as_\varphi(K_0).
\]
This proves (i).

\vspace{2mm}

Next, we show (iii), again following the  arguments in \cite{Giladietal}. The proof is similar to that of (i), but to apply the Blaschke selection theorem we first have to show that any body achieving the infimum lies inside of a larger body that we can draw a convergent subsequence from. Since $K$ is bounded, there exists $R>0$ such that $K\subset RB_n$. First, we show that for all $K^\prime\in\Cout(K)$ with $|K^\prime|\geq (Rn)^n|B_n|$, there exists $\tilde{K}\in\Cout(K)$ such that $\tilde{K}\subset RnB_n$ and $\as_\psi(K^\prime) \geq \as_\psi(\tilde{K})$. There exists an affine transformation $A:\R^n\to\R^n$ with $|\det(A)|=1$ and a positive number $t$ such that $tB_n$ has maximum volume among all ellipsoids contained in $A(K^\prime)$. By F. John's theorem \cite{John1948}, 
\begin{equation}\label{john}
tB_n \subset A(K^\prime) \subset tnB_n.
\end{equation}
Thus 
\[
(Rn)^n|B_n| \leq |K^\prime|=|A(K^\prime)| \leq (tn)^n|B_n|, 
\]
which implies $R\leq t$. Hence, 
\[
K \subset RB_n = \frac{R}{t}\cdot tB_n \subset \frac{R}{t}A(K^\prime).
\]
Now consider  the affine transformation $T=rA$ with $r=R/t\leq 1$, and set $\tilde{K}:=T(K^\prime)$. Then by \eqref{john}, $\tilde K\subset \frac{R}{t}(tnB_n)=RnB_n$ and $\tilde K\supset \frac{R}{t}(tB_n)=RB_n\supset K$. By Lemma \ref{scalinglem}, $\as_\psi(rK^\prime)\leq r^n\as_\psi(K^\prime)$. The $\text{SL}(n)$ invariance of $\as_\psi$ means that $\as_\psi(A(K))=\as_\psi(K)$ for any affine transformation $A$ with $|\det(A)|=1$. Putting all of this together, we obtain
\[
\as_\psi(\tilde K) = \as_\psi(rA(K^\prime))\leq r^n\as_\psi(A(K^\prime))= r^n \as_\psi(K^\prime)\leq \as_\psi(K^\prime).
\]
This proves the claim regarding those $K^\prime\in\Cout(K)$ with $|K^\prime|\geq (Rn)^n|B_n|$. 

Now suppose that $K^\prime\in\Cout(K)$ satisfies $|K^\prime| \leq (Rn)^n|B_n|$. Since $K^\prime$ has nonempty interior, there exists $r>0$ such that $rB_n\subset K^\prime$. For each $x\in K^\prime$, consider the cone $C_x$ with apex $x$ and base $x^\perp\cap rB_n$, and let $h_x$ denote the height of this cone. Then $C_x\subset K^\prime$, so
\[
|C_x| =\frac{h_x r^{n-1}}{n}|B_{n-1}|\leq |K^\prime| \leq (Rn)^n|B_n|
\]
which implies $K^\prime \subset \frac{n^{n+1}R^n|B_n|}{r^{n-1}|B_{n-1}|}B_n$.

Therefore, in these two cases we have shown that any convex body $K^\prime\in\Cout(K)$ that achieves the infimum must be contained in the ball  $\max\left\{Rn,\frac{n^{n+1}R^n|B_n|}{r^{n-1}|B_{n-1}|}\right\}B_n$. So assume that $K^\prime$ is contained in this ball. We can now use the argument from the proof of (i). By  Lemma \ref{isoperlem2}, \eqref{asball2} and \eqref{monotonepsi} we derive 
\begin{align*}
\as_\psi(K^\prime) \geq \as_\psi(\vrad(K^\prime)B_n)
\geq \as_\psi(\vrad(K)B_n).
\end{align*}
This shows that the infimum $\oSpsi(K)$ is finite. 
Now as before, there exists a sequence $\{C_k\}_{k\in\mathbb{N}}$ of convex bodies in $\Cout(K)$ such that for all $k\in\mathbb{N}$,
\[
\as_\psi(C_k) \leq \inf_{K^\prime\in\Cout(K)}\as_\psi(K^\prime)+\frac{1}{k} .
\]
Conversely, $\as_\psi(C_k)\geq \inf_{K^\prime\in\Cout(K)}\as_\psi(K^\prime)$ since  $C_k\in\Cout(K)$, so by another application of the squeeze theorem we deduce that 
\[
\lim_{k\to\infty}\as_\psi(C_k) = \oSpsi(K).
\]
By the Blaschke selection theorem, there exists a subsequence $\{C_{k_j}\}_{j\in\mathbb{N}}$ in $\Cout(K)$ that converges to a convex set $K_0\supset K$ with respect to the Hausdorff metric as $j\to\infty$. Since $K_0\supset K$, it has nonempty interior. Therefore, we may apply the lower semicontinuity \eqref{lowersemi} of $\as_\psi$  to derive
\[
\oSpsi(K) = \inf_{K\in\Cout(K)}\as_\psi(K^\prime)=\liminf_{j\to\infty}\as_\psi(C_{k_j}) \geq \as_\psi(K_0).
\]
Part (ii) now follows from (i) and the formula $\OSphistar(K)=\ISphi(K)$, and (iv) follows from (iii) and the formula $\iSpsistar(K)=\oSpsi(K^\circ)$.
\qed


\subsection{Proof of Proposition \ref{contlem}}

The argument is similar to that of \cite[Prop. 3.3]{Giladietal}. By hypothesis, $K$ has centroid at the origin, so there exists a Euclidean ball $\rho B_n$ of positive radius $\rho>0$ that is contained in $K$, and a sequence of convex bodies $\{K_\ell\}_{\ell\in\mathbb{N}}$  that have centroid at the origin and converge to $K$ in the Hausdorff metric. Thus, for every $\varepsilon>0$ there exists an integer $\ell_0$ such that for all $\ell\geq \ell_0$, 
\begin{equation*}
    K_\ell\subset K+\varepsilon B_n\quad \text{ and }\quad K\subset K_\ell+\varepsilon B_n.
\end{equation*}
For sufficiently small $\varepsilon$, we may assume that for all $\ell\geq \ell_0$, the inclusion $\frac{\rho}{10}B_n\subset K_\ell$ holds. The previous inclusions together imply that for all $\ell\geq\ell_0$,
\begin{equation}\label{inclusion1}
K_\ell \subset K+\varepsilon B_n =K+\frac{\varepsilon}{\rho}\cdot \rho B_n \subset K+\frac{\varepsilon}{\rho}K=\left(1+\frac{\varepsilon}{\rho}\right)K
\end{equation}
and
\begin{equation}\label{inclusion2}
K\subset K_\ell+\varepsilon B_n =K_\ell+\frac{10\varepsilon}{\rho}\cdot \frac{\rho}{10} B_n
\subset K_\ell+\frac{10\varepsilon}{\rho}K_\ell=\left(1+\frac{10\varepsilon}{\rho}\right)K_\ell.
\end{equation}
From \eqref{inclusion1}, the monotonicity property \eqref{mono1} of $\ISphi$ and Lemma \ref{scalinglem}, we get that for all $\ell\geq \ell_0$,
\begin{equation}\label{continc1}
    \ISphi(K_\ell) \leq \ISphi\left(\left(1+\frac{\varepsilon}{\rho}\right)K\right) =\sup_{K^\prime\in\Cin(K)}\as_\varphi\left(\left(1+\frac{\varepsilon}{\rho}\right)K^\prime\right)\leq \left(1+\frac{\varepsilon}{\rho}\right)^n\ISphi(K). 
\end{equation}
Similarly, from \eqref{inclusion2} we derive that for all $\ell\geq\ell_0$,
\begin{equation}\label{continc2}
    \ISphi(K) \leq 
    \left(1+\frac{10\varepsilon}{\rho}\right)^n\ISphi(K_\ell).
\end{equation}
The previous two inequalities now imply that for all $\ell\geq \ell_0$,
\[
\left(1+\frac{\varepsilon}{\rho}\right)^{-n}\ISphi(K_\ell)
\leq \ISphi(K)
\leq \left(1+\frac{10\varepsilon}{\rho}\right)^n\ISphi(K_\ell).
\]
Since $\varepsilon$ was arbitrary, claim (i) follows. 

Next, let $\varphi^*\in\fconcstardualgen$. Then for all $\ell\geq \ell_0$, inclusions \eqref{continc1} and \eqref{continc2} imply that  $K_\ell^\circ \supset\left(1+\frac{\varepsilon}{\rho}\right)^{-1}K^\circ$ and  $K^\circ \supset \left(1+\frac{10\varepsilon}{\rho}\right)^{-1}K_\ell^\circ$, respectively. Thus for all $\ell\geq \ell_0$, 
\begin{align*}
   \OSphistar(K_\ell) &=\ISphi(K_\ell^\circ) \geq \ISphi\left(\left(1+\tfrac{\varepsilon}{\rho}\right)^{-1}K^\circ\right)
    \geq \left(1+\tfrac{\varepsilon}{\rho}\right)^{-n}\ISphi(K^\circ)
    = \left(1+\tfrac{\varepsilon}{\rho}\right)^{-n}\OSphistar(K)
    \end{align*}
    and
    \begin{align*}
    \OSphistar(K) =\ISphi(K^\circ) &\geq \ISphi\left(\left(1+\tfrac{10\varepsilon}{\rho}\right)^{-1}K_\ell^\circ\right)\\
    &\geq \left(1+\tfrac{10\varepsilon}{\rho}\right)^{-n}\ISphi(K_\ell^\circ)
    = \left(1+\tfrac{10\varepsilon}{\rho}\right)^{-n}\OSphistar(K_\ell).
\end{align*} 
Therefore, for all $\ell\geq \ell_0$ we have
\[
\left(1+\frac{10\varepsilon}{\rho}\right)^{-n}\OSphistar(K_\ell)\leq \OSphistar(K) \leq \left(1+\frac{\varepsilon}{\rho}\right)^n \OSphistar(K_\ell).
\]
Since $\varepsilon>0$ was arbitrary, claim (ii) follows. The arguments to prove (iii) and (iv) are similar.
\qed
\subsection{Proof of Remark \ref{fnexample1}}\label{arctanexamplepf} 

For any $t>0$, we have $\varphi_m(t)>0$. The first two derivatives of $\varphi_m$ are:
\begin{align}
    \varphi_m^\prime(t) &= \frac{1}{2m}\left(t^{\frac{m-1}{m}}+t^{\frac{m+1}{m}}\right)^{-1}\\
    \varphi_m^{\prime\prime}(t) &=-\frac{1}{m^2}\left(t^{\frac{m-1}{m}}+t^{\frac{m+1}{m}}\right)^{-2}\left((m-1)t^{-\frac{1}{m}}+(m+1)t^{\frac{1}{m}}\right).
\end{align}
Since $\varphi_m^{\prime\prime}(t)<0$ for all $t>0$, the function $\varphi_m:(0,\infty)\to(0,\infty)$ is concave. By continuity, $\lim_{t\to 0}\varphi_m(t)=\arctan(0)=0$ and since $\varphi_m(t)$ is monotonically increasing and bounded above by $\pi/2$,
\[
0\leq \lim_{t\to\infty}\frac{\varphi_m(t)}{t}\leq\lim_{t\to\infty}\frac{\pi}{2t}=0.
\]
Furthermore, $\varphi_m(0)=0$, which shows that $\varphi_m\in\fconc$. 

To prove  that $\varphi_m\in\fconcstargen$, it remains to show that $\varphi_m(t)/\sqrt{t}$ is decreasing. We have  $\frac{d}{dt}(\varphi_m(t)/\sqrt{t})<0$ if and only if $\varphi_m^\prime(t)<\frac{\varphi_m(t)}{2t}$, which is equivalent to
\begin{equation}\label{finalineq}
    t^{1/m}+t^{-1/m}>\left[\frac{m}{2}\cdot\arctan(t^{1/m})\right]^{-1}.
\end{equation}
We show that $(\arctan(x))^{-1}<x+\frac{1}{x}$ for any $x>0$, which is equivalent to    $\arctan(x)>\frac{x}{x^2+1}$ for $x>0$. Let  $g(x):=\arctan(x)-\frac{x}{x^2+1}$. Then  $g^\prime(x)=\frac{2x^2}{(x^2+1)^2}>0$ for any $x>0$, so $g$ is  increasing on $(0,\infty)$. Thus for any $x>0$ we have  $g(x)>g(0)=0$, which proves the inequality. Taking $x=t^{1/m}$ and using the fact that $m\geq 2$, we derive
\[
\left[\frac{m}{2}\cdot\arctan(t^{1/m})\right]^{-1} = \frac{2}{m}(\arctan(t^{1/m}))^{-1}\leq (\arctan(t^{1/m}))^{-1}  < t^{1/m}+t^{-1/m}.
\]


\subsection{Remarks on the Centroid Assumption}\label{centroidsec}

We defined $\Cin(K)$ and $\Cout(K)$ to include only bodies with centroid at the origin so that we can apply the $L_\varphi$ and $L_\psi$ affine isoperimetric inequalities in Lemmas \ref{isoperlemphi} and \ref{isoperlem2}, respectively. This, in turn, allows us to  state our results for the broad range of functions in    $\fconcstargen\cup\fconcstardualgen$ and all of the functions in $\fconv$.

Ye \cite{Deping2014Steiner} used Steiner symmetrizations to prove the following $L_\varphi$ and $L_\psi$ affine isoperimetric inequalities.  These inequalities do not have the restriction that the centroid of the body is the origin; on the other hand, they impose  some additional restrictions  on the functions $\varphi$ and $\psi$, and in the former case, on the curvature of $K$. 

\begin{lemma}\cite[Thm. 3.2]{Deping2014Steiner}\label{isoperlemye}
Let $K\in\K$ and suppose that $\varphi\in\fconc$ is such that the function $F_n(t)=\varphi(t^{n+1})$ for $t\in(0,\infty)$ is concave. Then
\begin{align*}
    \as_\varphi(K) &\leq \as_\varphi(\vrad(K)B_n). \end{align*}
Moreover, if $F_n(\cdot)$ is strictly concave and $K$ has positive Gaussian curvature almost everywhere, then equality holds if and only if $K$ is an origin-symmetric ellipsoid.
\end{lemma}


\begin{lemma}\cite[Thm. 3.4]{Deping2014Steiner}\label{isoperlemye2}
Let $K\in\K$ and suppose that $\psi\in\fconc$ is such that the function $G_n(t)=\psi(t^{n+1})$ for $t\in(0,\infty)$ is convex. Then
\begin{align*}
    \as_\psi(K) &\geq \as_\psi(\vrad(K)B_n). \end{align*}
Moreover, if $G_n(\cdot)$ is strictly convex, then  equality holds if and only if $K$ is an origin-symmetric ellipsoid.
\end{lemma}

Thus, one could remove the assumption that the bodies in $\Cin(K)$ or $\Cout(K)$ have centroid at the origin by using Lemmas \ref{isoperlemye} and \ref{isoperlemye2} in all of the proofs instead of Lemmas \ref{isoperlemphi} and \ref{isoperlem2}, but the downside is then one captures  fewer functions from $\fconc$ and $\fconv$. 
Furthermore, the assumption  in the definitions of $\Cin(K)$ and $\Cout(K)$ that the bodies have centroid at the origin is tacitly assumed in the definition of the extremal $L_p$ affine surface areas in \cite{Giladietal}. This is because the definition \eqref{asp} of $L_p$ affine surface area is  stated for convex bodies with centroid at the origin. Ultimately, it is an open problem to  
obtain the results of this article for a richer class of functions than $\fconcstargen\cup\fconcstardualgen$.


\section*{Acknowledgments}
The author would like to thank  Elisabeth Werner and Deping Ye for the discussions, and the anonymous referee for carefully reading this manuscript and providing valuable comments.


\bibliographystyle{plain}
\bibliography{main}

\vspace{3mm}

\noindent {\sc Department of Mathematics \& Computer Science, Longwood University, 23909}

\noindent {\it E-mail address:} {\tt hoehnersd@longwood.edu}

\end{document}